\documentclass[letterpaper, USenglish]{amsart}
\usepackage[latin1]{inputenc}
\usepackage[table, dvipsnames]{xcolor}
\usepackage{amsmath}
\usepackage{amsfonts}
\usepackage{amssymb}
\usepackage{amsthm}
\usepackage{graphicx}
\usepackage{mathtools}
\usepackage{verbatim}
\usepackage{hyperref}
\usepackage{array}
\usepackage{float}
\usepackage{tikz}
\usepackage{todonotes}
\usepackage{thmtools}
\usepackage{indentfirst}
\usepackage{thm-restate}
\usepackage{isodate}
\usepackage[linesnumbered,ruled,noend]{algorithm2e}

\newcommand{\commentR}[1]{\tcp*[r]{\parbox[t]{4.5cm}{\raggedright\normalfont{#1}}}}
\newcommand{\commentF}[1]{\tcp*[f]{\parbox[t]{4.5cm}{\raggedright\normalfont{#1}}}}
\SetKwComment{tcp}{}{}
\let\oldnl\nl
\newcommand{\nonl}{\renewcommand{\nl}{\let\nl\oldnl}}
\makeatother
\usepackage{etoolbox} 
\makeatletter
\patchcmd{\algocf@makecaption@ruled}{\hsize}{\textwidth}{}{} 
\patchcmd{\@algocf@start}{-1.5em}{0em}{}{}
\makeatother
\SetAlCapHSkip{0pt} 
\newcommand{\adj}{\textup{adj\,}}
\newcommand{\bigo}{\mathcal{O}}
\newcommand{\Z}{\mathbb{Z}}
\newcommand{\Q}{\mathbb{Q}}
\newcommand{\R}{\mathbb{R}}
\newcommand{\lat}{\mathcal{L}}
\newcommand{\N}{\mathbb{N}}
\newcommand{\Mn}{\textup{M}_n}
\newcommand{\SAP}{\texttt{SAP\,}}

\newcommand{\SIAP}{\texttt{SIAP\,}}

\newcommand{\CAP}{\texttt{CAP\,}}

\newcommand{\mb}[1]{\mathbf{#1}}

\newcommand\norm[1]{\left\lVert #1 \right\rVert}
\newcommand\opnorm[1]{\left\lVert #1 \right\rVert_{\text{op}}}

\tikzstyle{line} = [draw, thick, color=black]

\definecolor{bubbles}{rgb}{0.91, 1.0, 1.0}
\definecolor{columbiablue}{rgb}{0.61, 0.87, 1.0}
\definecolor{lightcyan}{rgb}{0.88, 1.0, 1.0}

\newtheorem{theorem}{Theorem}
\newtheorem{proposition}[theorem]{Proposition}
\newtheorem{lemma}[theorem]{Lemma}
\newtheorem{corollary}[theorem]{Corollary}
\theoremstyle{definition}
\newtheorem*{definition*}{Definition}
\declaretheorem[name=Definition]{definition}

\numberwithin{theorem}{section}
\numberwithin{proposition}{section}
\numberwithin{lemma}{section}
\numberwithin{corollary}{section}
\numberwithin{definition}{section}
\numberwithin{example}{section}
\numberwithin{equation}{theorem}
\numberwithin{figure}{section}
\numberwithin{table}{section}

\title{Simultaneous Approximation for Lattice-based Cryptography}
\author{Julia VanLandingham}

\date{\today}

\begin{document}

\thanks{This research was supported by NSF award DMS-2336000 under the supervision of PI Dr. Daniel E. Martin.}

\begin{abstract} 
     We define two new problems called SIAP and CAP related to solving SIVP and CVP in a subset of lattices called Simultaneous Approximation (SA) lattices.
     We give dimension- and gap-preserving, deterministic polynomial-time and space reductions from SVP$_\gamma$, SIVP$_\gamma$, and CVP$_\gamma$ to their corresponding problems in  SA lattices. These reductions show that instances of these problems in SA lattices are just as hard as general instances and thus are interesting new problems to consider for use in cryptography. We also show that the reductions are optimal in regards to integer inflation.
\end{abstract}
\maketitle

\thanks

\section{Introduction}\label{intro_section}

Over the past two decades, lattice problems have risen in popularity for cryptographic use.
Lattice-based systems provide several benefits: average to worst case reduction \cite{worst_to_average}, accepted quantum and classical security, and homomorphic encryption capability \cite{homomorphic}. 
However, naive implementations require large public key sizes. Our goal is to study lattice problems that provide the ability for smaller key sizes for systems based on them. 

In 2006, Lyubashevsky and Micciancio introduced ideal lattices into cryptography \cite{ideal_lattices}. 
These lattices provide ability for smaller public key sizes, but have no guarantee that their instances of lattice problems are as hard as  general instances. 
Campbell et al. outlined an algorithm in \cite{soliloquy} (verified in \cite{mildly_short}) for specific ideal lattices that finds an approximation of the shortest vector by a factor $\exp \big(\tilde{\bigo}(\sqrt{m})\big)$, where $m$ is a parameter of the lattice, while the best known polynomial-time generic lattice algorithms have an approximation factor $\exp \big(\tilde{\bigo}(m)\big)$.
Other indications that ideal lattices may not be as classically or quantum secure as general lattices can be seen in \cite{multiquadratics,Twisted_PHS,soliloquy,recover_short_gen,stickelberger,novoselov,ideal_svp}  and 
\cite{unit_group,mildly_short,quantum_SVP_ideals}, respectively.

Simultaneous Diophantine Approximation studies the problem of approximating a real vector by a rational vector with entries that have a common denominator.
While Diophantine Approximation problems have been successfully used in cryptanalysis (see \cite{Garba,wiener}), minimal research has been done towards basing cryptosystems on them (see \cite{inhomo_DA,DA-encrypt}).
In 1982, Lagarias considered a version of Simultaneous Diophantine Approximation, called the Good Diophantine Approximation problem (GDA), that includes a bound on a solution's denominator and gave a reduction from GDA to the approximate Shortest Vector Problem (SVP) \cite{lagarias}.
In 2020, Martin considered the version of Simultaneous Approximation that does not impose any bound on the denominator. He proposed a new class of lattices related to this problem and called them Simultaneous Approximation (SA) lattices \cite{daniel}.
An SA lattice of dimension $n$ can be described by $n+1$ integers rather than the $n^2$ needed in general. 
Solving SVP in these SA lattices can be viewed as solving the Simultaneous Approximation Problem (SAP) (see Section \ref{approximate_section}).
Using SA lattices, Martin gave a polynomial-time and space reduction from approximate SVP to SAP. He also gave reductions from GDA to SAP, and vice versa, to show that approximate SVP is equivalent to SAP.
Unfortunately, this does not guarantee that the public key size for an SA lattice-based system will be smaller than for a general-lattice based system of the same security level. It does show that solving SVP in an SA lattice is just as hard as in a general lattice - something ideal lattices cannot claim. 

In the next section, we define two new SA lattice problems called SIAP and CAP that can be viewed as solving SIVP and CVP, respectively, in an SA lattice.
We adapt Martin's algorithm to demonstrate equivalence of these problems with their generic counterparts. 
This adaptation remains dimension- and gap-preserving, but significantly decreases integer inflation, which is a critical factor in bounding key sizes needed to maintain security in future SA lattice-based cryptosystems.

More specifically, we prove the following: 
\begin{theorem}
    Algorithms \ref{SVP_reduction}, \ref{SIVP_reduction}, and \ref{CVP_reduction} are dimension- and gap-preserving, deterministic polynomial-time reductions from SVP$_\gamma$, SIVP$_\gamma$, and CVP$_\gamma$, to SAP$_\gamma$, SIAP$_\gamma$, and CAP$_\gamma$, respectively, with input length scaled by $\bigo\big(n^2\log(n)\big)$. 
\end{theorem}

This work provides a significantly better bound on integer inflation than Martin's and we show in Section \ref{optimal_section} that it is optimal.
These reductions also immediately give the following:

\begin{theorem}
SIAP$_\gamma$ and CAP$_\gamma$ are NP-hard for any constant $\gamma$ under the $\ell_1,\ell_2$, or $\ell_\infty$-norm. SAP$_\gamma$ is NP-hard under the $\ell_\infty$-norm for any constant $\gamma$ and is NP-hard under randomized reductions for the $\ell_1$ or $\ell_2$-norm with constant $\gamma< 2$ or $\gamma < 2^{1/2}$, respectively. 
\end{theorem}

Each  reduction utilizes a method of approximating a general lattice by an SA lattice which we describe colloquially in Section \ref{intuition_section} and analyze throughout Section \ref{approximate_section}.
We give reductions from the general to SA specific cases of SVP, SIVP, and CVP in Sections \ref{SVP_section}, \ref{SIVP_section}, and \ref{CVP_section}, respectively, and discuss the optimality of the reductions in Section \ref{optimal_section}.

\section{Approximating with SA Lattices}\label{approximate_section}

\begin{definition}\label{SA_def}
    A \emph{Simultaneous Approximation  (SA) lattice} of dimension $n$ is a lattice generated by the columns of $I_n$ and a vector $\mb{x}\in\Q^n$.
\end{definition}

Our goal is to study three classical lattice problems in the restricted setting of SA lattices. For the remainder of the paper, let $\norm{\cdot}$ be some $\ell_p$-norm and $\min^\times$ represent the minimum excluding 0.
Recall that the approximate Shortest Vector Problem (SVP$_\gamma$) is to find a nonzero lattice vector at most $\gamma$ times longer than the shortest possible. 
Also recall that the fractional part of a real number $x$ is represented as $\{x\}$. For $\mb{x} \in \Q^n$, we will take $\{\mb{x}\}$ to mean the vector of element-wise fractional parts. Consider what it means to solve SVP$_\gamma$ in an SA lattice.

\begin{definition}
    Given $\gamma\in \Q$ and  $\mb{x}\in\Q^n$, a solution to the \emph{approximate Simultaneous Approximation Problem} (SAP$_\gamma$) is an  output $b_0\in\Z$ such that $0<\norm{\{b_0\mb{x}\}}\leq \gamma\cdot\min^{\times}_{b\in\Z}\left\{ \norm{\{b\mb{x}\} }\right\}.$ Let $\SAP(\gamma, \mb{x})$ denote an oracle for such a problem.
\end{definition}

To find a short vector in an SA lattice we need to find $a_i,b\in \Z$ such that $a_1\mb{e}_1+\cdots+a_n\mb{e}_n - b\mb{x} \approx 0$ where $\mb{e}_i$ is the $i^{\text{th}}$ column of $I_n$.
This is the same as solving SAP$_\gamma$ with an input of $\mb{x}$ and $b$ the output since $a_1\mb{e}_1+\cdots+a_n\mb{e}_n - b\mb{x}=\{b\mb{x}\}.$
Thus, solving SVP$_\gamma$ in an SA lattice can be viewed as solving SAP$_\gamma$. This relationship is where these SA lattices get their name.

Recall that the $i^{\text{th}}$ successive minimum of a lattice, denoted $\lambda_i$, is the smallest radius $r$ such that the closed ball $B(0,r)$ contains $i$ linearly independent lattice vectors. 
The approximate Short Independent Vectors Problem (SIVP$_\gamma$) is to find a set of $n$ independent lattice vectors that all are at most $\gamma$ times longer than $\lambda_n$. 
The following new problem can be viewed as solving SIVP$_\gamma$ in an SA lattice:

\begin{definition}
    Given $\gamma\in \Q$ and $\mb{x}\in\Q^n$, a solution to the \emph{approximate Shortest Independent Approximation Problem} (SIAP$_\gamma$) is an output $b_1,b_2,\dots,b_n\in \Z$ such that $\{b_i\mb{x}\}$ are linearly independent and $\max_i \left\{\norm{\{b_i\mb{x}\}}\right\}\leq \gamma\lambda_n.$
    Let $\SIAP(\gamma, \mb{x})$ denote an oracle for such a problem.
\end{definition}

Recall that the approximate Closest Vector Problem (CVP$_\gamma$) is to find a lattice vector that is within $\gamma$ times the shortest distance between a lattice vector and some given target vector.
We define the following new problem that can be viewed as solving CVP$_\gamma$ in an SA lattice:

\begin{definition}
Given $\gamma \in\Q$, $\mb{x}\in\Q^n,$ and $ \mb{t}\in\Q^n$, a solution to the \emph{approximate Closest Approximation Problem} (CAP$_\gamma$) is an output $b_0\in\Z$  such that $ 0\leq\norm{\{b_0\mb{x}-\mb{t}\}}\leq \gamma \min_{b\in\Z}\left\{\norm{\{b\mb{x}-\mb{t}\}}\right\}$. Let $\CAP(\gamma, \mb{x}, \mb{t})$ denote an oracle for such a problem.
\end{definition}

We first outline an algorithm for approximating a given lattice with an SA lattice. 

\subsection{Intuition}\label{intuition_section}
Let $\lat$ be a full-rank lattice such that $\lat = M\Z^n$ for some $M\in \Mn(\Z)$. 
Suppose we had an integer $c$, a matrix $A\in \Mn(\Z)$ with entries small  relative to  $c\cdot\det(M)$, and some vector $\mb{b}\in\Z^n$ such that the columns of $c\cdot\adj M +A$ along with $\mb{b}$ generate $\Z^n$.
Then the columns of $M(c\cdot\adj M+A)$ along with $M\mb{b}$ generate $\lat$.
Notice that $M(c\cdot\adj M +A) = c\cdot\det(M) I_n + MA$ is a \emph{nearly scaled orthonormal} matrix of lattice vectors because the entries of $A$ are small relative to $c\cdot\det(M)$. 
If we multiply everything by $(c\cdot\det(M) I_n + MA)^{-1}$, smallness will be preserved, and we get an SA lattice generated by the columns of $I_n$ and $(c\cdot\det(M) I_n + MA)^{-1} M\mb{b}$, which plays the roll of $\mb{x}$ in Definition \ref{SA_def}. 
So, if such an integer $c$, matrix $A$, and vector $\mb{b}$ can be found, then we can  construct an SA lattice that approximates $\lat$.

Say $\tilde{M}=c\cdot\adj M+A.$ In order for the columns of $\tilde{M}$ and $\mb{b}$ to generate $\Z^n$, it suffices to ensure that replacing a column of $\tilde{M}$ with $\mb{b}$ results in a matrix of determinant $\pm 1$, since the columns of such a matrix generate an index 1 subgroup of $\Z^n$.
By Cramer's rule, this is equivalent to ensuring that $\adj(\tilde{M})\mb{b}$ has some entry that is a 1.
Note that if one row of $\adj(\tilde{M})$ has entries that are collectively coprime, then it is simple to find a $\mb{b}$ vector by using appropriate B{\'e}zout coefficients.  For simplicity, we will make the first two entries of the last row of $\adj(\tilde{M})$ coprime. Let us consider an example of how.

Consider a matrix 
$M \in \textup{M}_4(\Z).$ Notice that the elements in question correspond to the determinants of submatrices constructed by removing the last column and either the first or second row. 
We find the elements of $A$, denoted $a_{i,j}$ below,  iteratively, ensuring at the $i^{\text{th}}$ step that the determinants of the  top left $i\times i$ minors of these submatrices are coprime.

The result after this process is a matrix 
$$\tilde{M} = c\cdot\adj M + A = \begin{pmatrix}
    *+a_{1,1}& * &* &*\\
    *+a_{1,2} & * &  *& *\\
    * & *+a_{2,3} & * & *\\
    * & * & *+a_{3,4} & *
\end{pmatrix}$$
where, for example, $a_{2,3}$ is chosen to make the determinants below coprime - a condition that guarantees a simularly chosen $a_{3,4}$ exists at the next (and final) stage.
$$\det\begin{pmatrix}
    -& - &| &|\\
    \textcolor{red}{*+a_{1,2}} & \color{red}{*} &  |& |\\
    \color{red}{*} & \color{red}{*+a_{2,3}} & | & |\\
    - & - & | & |
\end{pmatrix} 
\text{ and }\det\begin{pmatrix}
    \textcolor{red}{* +a_{1,1}}& \color{red}{*} & |& |\\
     -& - &| &|\\
    \color{red}{*} & \color{red}{*+a_{2,3}} & | & |\\
    - & - & | & |
\end{pmatrix}. $$

A  similar methodology was proposed by Martin in \cite{daniel} and used to show his reduction from approximate SVP to SAP.
Martin used matrices and vectors with entries from  polynomial rings  rather than $\Z$, which avoided the need to consider what is called Jacobsthal's function in the worst-case analysis. 
Jacobsthal's function and its relevance to this work are discussed in Section \ref{inflation_section}.
Our switch to using integer entries gives an improved bound on integer input lengths throughout the algorithm of
$\bigo\big(n^2\log(nk)\big)$ from Martin's $\bigo\big(n^4\log(nk)\big)$. We show our bound is optimal in Section \ref{optimal_section}.

\subsection{Algorithm}
For a matrix $M \in \Mn(\Z)$ and $i,j=1,\dots, n$, let $M^{(i)}$ denote the top left $i\times i$ minor of $M$ and let $(M)_{i,j}$ denote the element at row $i$ and column $j$.

\begin{figure}[H]
\begin{algorithm}[H]
\DontPrintSemicolon
\caption{Approximating a general lattice by an SA lattice}\label{approximate}
    \KwIn{ $(k,M)$ such that $k\in \N$ and $M\in \Mn(\Z) \,(n\ge 8)$ with $ \det(M)\not = 0$ and $k\geq \max_{i,j} |(M)_{i,j}|$}
\KwOut{ ($\mb{x},\tilde{M})$ with $\tilde{M}\in \Mn(\Z)$ and $\mb{x}\in\Q^n$ such that $M\tilde{M}\mb{x}$ and the columns of $M\tilde{M}$ generate $M\Z^n$}
    $\tilde{M}\gets  1728(nk)^{3n+15}\adj M$ \\
    $B_1 \gets$ $\tilde{M}$ with row 1 and column $n$ removed\\
     $B_2\gets$ $\tilde{M}$ with row 2 and column $n$ removed\\

    \If(\commentF{$\triangleright$ make elements different}){$(B_1)_{1,1}=(B_2)_{1,1}$}{
        $(B_1)_{1,1} \gets (B_1)_{1,1} +1$\\
        $(\tilde{M})_{2,1}\gets (\tilde{M})_{2,1} +1$
    }
        
   \For{$i=1,2,\dots, n-1$}{
        \While(\commentF{$\triangleright$ iterate until coprime}){$\gcd\Big(\det\big(B_1^{(i)}\big),\det\big(B_2^{(i)}\big)\Big) \not= 1$}{
             $(B_1)_{i,i}\gets (B_1)_{i,i} + 1$\\
             $(B_2)_{i,i}\gets (B_2)_{i,i} + 1$\\
            \If(\commentF{$\triangleright$ update top left entry also}){$i=1$}{
                 $(\tilde{M})_{1,1}\gets (\tilde{M})_{1,1} + 1$
            }
             $(\tilde{M})_{i+1,i}\gets (\tilde{M})_{i+1,i} + 1$
        }
    }

     $b_1,b_2\gets$ integers such that $|b_1|\leq |\det (B_2)|$ and\linebreak$\det(B_1)b_1 + \det(B_2) b_2 =1$ \commentR{$\triangleright$ Euclidean Algorithm}
   	\Return ($\tilde{M}^{-1}(b_1,b_2,0,\dots,0)^T$, $\tilde{M}$)
\end{algorithm}
\end{figure}

The following lemma will be crucial in our proof of Algorithm \ref{approximate}:

\begin{lemma}\label{solution_lemma}
For $a,b,c,d\in \Z,$ if $d$ and $b$ are coprime, there exists an $x\in\Z$ such that $a+bx$ and $c+dx$ are coprime.
\end{lemma}

\begin{proof}
    Let $s,t$ be such that $bs+dt=1$ and consider
    $$\begin{pmatrix}
        d &-b\\
        s & t
    \end{pmatrix}\begin{pmatrix}a+bx\\c+dx\end{pmatrix} = \begin{pmatrix}
        da-bc\\
        sa+tc+x
    \end{pmatrix}.$$
    
    Since the determinant of the left matrix is 1, then $da-bc$ and $sa+tc+x$ are coprime if and only if $a+bx$ and $c+dx$ are coprime. We know know such an $x$ exists to make  $da-bc$ and $sa+tc+x$ coprime.
\end{proof}

We now prove the correctness of Algorithm \ref{approximate}.

\begin{proposition}\label{approximate_prop}
    Let $M$ be as in the input of Algorithm \ref{approximate} and $\tilde{M}$ be as in the output of Algorithm \ref{approximate}. Algorithm  \ref{approximate} returns a vector $\mb{x}$ such that $M\tilde{M}\mb{x}$ and the columns of $M\tilde{M}$ generate $M\Z^n$.
\end{proposition}

\begin{proof}
It suffices to show that the algorithm will result in $\det(B_1)$ and $\det(B_2)$ coprime. In lines 4-6 we ensure that $(B_1)_{1,1}\not =(B_2)_{1,1}$, so we can always perturb both elements in the first iteration of the following \textbf{for} loop to be coprime.
We must update $(\tilde{M})_{1,1}$ and $(\tilde{M})_{i+1,i}$, unlike all other iterations, and do so in lines 11-12.
Say the total perturbation added in iteration $i$ of the \textbf{for} loop on lines 7-13 is called $x_i.$ 
When $x_i$ is added, we have the following relationship between the new and old determinants:
$$\det\Big(B_{1 \,(\text{new})}^{(i)}\Big) = \det\Big(B_{1 \,(\text{old})}^{(i)}\Big) + \det\Big(B_{1}^{(i-1)}\Big)x_i.$$
Likewise, $$\det\Big(B_{2 \,(\text{new})}^{(i)}\Big) = \det\Big(B_{2 \,(\text{old})}^{(i)}\Big) + \det\Big(B_{2}^{(i-1)}\Big)x_i.$$ 
So, it is guaranteed by the previous iteration of the \textbf{for} loop that $\det\big(B_1^{(i-1)}\big)$ and $\det\big(B_2^{(i-1)}\big)$ are coprime. By Lemma \ref{solution_lemma}, we know for every step $i=2,3,\dots, n-1$ that $x_i$ exists and thus the \textbf{while} loop on lines 8-13 will terminate. Hence, after the termination of the \textbf{for} loop, $\det(B_1)$ is coprime to $\det(B_2)$.
\end{proof}

\subsection{Integer Inflation} \label{inflation_section}
The following algebraic lemma will be useful.

\begin{lemma}\label{lemma:log-bound}
    Let $n,k,c\in\N$ with $n\geq 8$ and $k,c\geq 1$. Then $$c(nk)^{n}-n^2\log^2(cn^{n+1}k^{n}) \geq 0.$$
\end{lemma}

\begin{proof}
    First notice that $\log(x) < x^{1/3}$ for all $x\geq 94$ by simple calculus. Then it follows directly by algebraic manipulations that the statement holds for all $n\geq 8$ and  $k,c\geq 1$
\end{proof}

Recall Hadamard's inequality \cite{hadamard}:

\begin{theorem}[Hadamard's Inequality]\label{hadamard}
    For a matrix $M\in\Mn(\R)$ with entries bounded in magnitude by $k\geq 1$, $\lvert \det(M)\rvert\leq k^n n^{n/2}$.
\end{theorem}

This proves vital since we will work with the adjugate matrix whose entries are determinants of $(n-1)\times(n-1)$ minors of the original. We also need Jacobsthal's function:

\begin{definition}
    For $n\in\mathbb{N}$, \emph{Jacobsthal's function}, denoted $j(n)$, is the smallest $m\in\N$ such that every sequence of $m$ consecutive integers contains an integer coprime to $n.$
\end{definition}

This means that Jacobsthal's function gives an upper bound on the perturbation required for two elements to be coprime. Thus, Jacobsthal's function also gives an upper bound on the number of iterations of the \textbf{while} loop beginning on line 8 of Algorithm \ref{approximate}. In 1978, Iwaniec proved the following important theorem about this function \cite{jacobsthal}:

\begin{theorem}\label{jacobsthals}
    We have $j(n) = \bigo\big(\log^2(n)\big).$ 
\end{theorem}

This asymptotic bound is extremely useful in many contexts, but does not reflect what we expect on average. 
In fact, the probability that two random integers $a,b$ are  coprime is $6/\pi^2$ \cite{coprime}.
Given $a,b\in \N$, we can compute an upper bound on the 
expected value $E[x]$ of $x\in\N$ such that $\gcd(a+x,b)=1$:
\begin{align*}
    E[x] &\leq 0\cdot\frac{6}{\pi^2} + 1\cdot (1-\frac{6}{\pi^2})\frac{6}{\pi^2} + 2\cdot (1-\frac{6}{\pi^2})^2\frac{6}{\pi^2}+\cdots\\
    & = \frac{6}{\pi^2}\sum_{i=1}^\infty i\cdot(1-\frac{6}{\pi^2})^i.
\end{align*}
By taking the derivative of a geometric series and multiplying by $\left(1-6/\pi^2\right)$, we find this is equal to $$\dfrac{6}{\pi^2}\cdot\dfrac{1-\frac{6}{\pi^2}}{(\frac{6}{\pi^2})^2} = \dfrac{\pi^2}{6}-1 \approx 0.645.$$
Experimental data shows that $x$ is 0 about 60.8\% of the time, 1 about 28.6\% of the time, and 2 about 4.2\% of the time. On average, $x\approx0.576$.

Let us now consider the integer inflation throughout Algorithm \ref{approximate}.

\begin{lemma}\label{inflation}
    Let $k$ and $M$ be as in the input of Algorithm \ref{approximate} and $c\in \N$ with $c\ge 1$. Then, the magnitudes of the entries of the matrix $\tilde{M}$ in Algorithm  \ref{approximate} are bounded above by $c(2nk)^{n}$ at every step of the algorithm.
\end{lemma}

\begin{proof}
    By using the Bariess algorithm \cite{bariess} and Hadamard's inequality, we can ensure that the magnitude of entries of $\adj M$ are bounded above by $k^{n-1}(n-1)^{(n-1)/2}$ at every intermediate step of the computation. So after line 1, the magnitudes of the entries of $\tilde{M} = c\cdot\adj M$ are  bounded above by $ck^nn^{n/2} \leq c(nk)^n.$ 
    Let $s_i$ be an upper bound on the magnitude of the entries of $\tilde{M}$ after iteration $i$ of the \textbf{for} loop on lines 7-13. Then $s_0 = c(nk)^n.$

    In any iteration $i=1,2,\dots, n-1$, we take determinants of appropriately sized minors and then perturb them to be coprime.
    By Hadamard's inequality, the determinants prior to perturbation have magnitude bounded above by $(s_{i-1}n)^n$. Since Jacobsthal's function gives an upper bound on the size of the perturbation necessary, we can apply Theorem \ref{jacobsthals} to get that the magnitude of the perturbation applied in iteration $i$ is $\bigo\big(\log^2\big((s_{i-1}n)^n\big)\big)$.
    So, $$s_i\leq s_{i-1}+n^2\log^2(s_{i-1}n).$$
    
    We can see for all $i=1,2,\dots, n-1$, that $s_{i-1}\leq s_i$ since we are adding positive values to entries at every iteration.
    Since, by Lemma \ref{lemma:log-bound}, $n^2\log^2(s_0n)\leq s_0$ for all $n\geq 2,\, k,c\geq 1$, and $s_{i-1}\leq s_i$, then $n^2\log^2(s_in)\leq s_i$ for $i=1,2,\dots, n-1$.
    Hence, for each $i=1,2,\dots, n-1,$ $s_i\leq 2s_{i-1}.$

    At the termination of the \textbf{for} loop, the magnitude of the entries are bounded above by \begin{equation*}s_{n-1}\leq 2^{n-1}s_0=2^{n-1}c(nk)^{n}\leq c(2nk)^{n}.\qedhere\end{equation*}
\end{proof}

\begin{theorem}\label{overall_inflation}
    Let $k$ and $M$ be as in the input of Algorithm \ref{approximate}, $\tilde{M}$ be as in the output of Algorithm \ref{approximate}, and $c$ be the multiplier used in line 1. The bitlength of the entries of the vector output from Algorithm  \ref{approximate} is $\mathcal{O}\big(n^2\log(nk)\big)$.
\end{theorem}

\begin{proof}
    By Lemma \ref{inflation}, we know that the magnitude of the entries of $\tilde{M}$ are bounded above by $c(2nk)^n=1728(nk)^{3n+15}(2nk)^n\le1728(2nk)^{4n+15}$
    
    By Hadamard's inequality, the entries of $\tilde{M}^{-1}$ are bounded above by $$(1728(2nk)^{4n+15})^n n^{n/2}\leq 1728^n(2nk)^{4n^2+16n}.$$
    Line 14 and Hadamard's inequality gives 
    $$|b_1|\le |\det(B_1)|\le( 1728(2nk)^{4n+15})^nn^{n/2}\le1728^n(2nk)^{4n^2+16n}.$$ 
    The relationship on line 14 gives, $|b_2|\le1728^n(2nk)^{4n^2+16n}.$
    
    Putting these together, the magnitude of the entries of the output vector will be at most 
    $$2\big( 1728^n(2nk)^{4n^2+16n} \cdot 1728^n(2nk)^{4n^2+16n}\big)\le1728^{2n}(2nk)^{8n^2+32n+1}$$ since only two entries of $(b_1,b_2,0,\dots,0)^T$ are nonzero.
    So, the bitlength of the entries of the output vector is $\bigo\big(n^2\log(nk)\big).$
\end{proof}

This bound may be possible to improve in the average since it is likely that the perturbations $a_{i,j}$ will not often achieve the worst case bound. In fact, experiments show that $a_{i,j}$ is 0 about 55\% of the time, 1 about 31\% of the time, 2 about 10\% of the time, and 3 about 9\% of the time. On average, $a_{i,j}\approx0.92$.
Taking the adjugate at the beginning and the inverse at the end both inflate the magnitude of entries by a power of $n$. Without some method of avoiding these steps, even if average case analysis were completed, the bitlength of the entries of the output vector  would be at best $\bigo\big(n^2\log(k)\big)$, which is not significantly less than Theorem \ref{overall_inflation}.

\subsection{Gap-Preservation}
 We start by deriving bounds on the operator norm of a scaled version of $M\tilde{M}.$ 

\begin{definition}
    For a linear operator $M$, the \emph{operator norm} of $M$, denoted as  $\opnorm{M}$, is equal to $\displaystyle\smash{\max_{\norm{\mb{u}}=1}} \norm{M\mb{u}}$.
\end{definition}

\begin{lemma}\label{opnorm_prop}
    Let $c\in \N$, $k$ and $M$ be as in the input of Algorithm \ref{approximate}, and $\tilde{M}=c\cdot\adj M +A$ be as in the output of Algorithm \ref{approximate}. Then,
    $$1-\frac{kn^4\log^2(2nk)+kn^2\log^2(c)}{|c\cdot\det M|}\leq \norm{I_n+\frac{MA}{c\cdot\det M}} \leq 1+\frac{kn^4\log^2(2nk)+kn^2\log^2(c)}{|c\cdot\det M|}.$$
\end{lemma}

\begin{proof}
    Let $\mb{u}$ be such that $\norm{\mb{u}}=1.$  Then
    
    $$
    \norm{\left(I_n+\frac{MA}{c\cdot\det M}\right)\mb{u}} = \norm{I_n\mb{u}+\frac{MA}{c\cdot\det M}\mb{u}}
    \leq 1 + \frac{\norm{MA\mb{u}}}{|c\cdot\det M|}.
    $$
    Lemma \ref{inflation} gives that the entries of $\tilde{M}$ are bounded above by $c(2nk)^n$. We know that the entries of $A$ are the perturbations added  to $\tilde{M}$ throughout Algorithm \ref{approximate} and, by Theorem \ref{jacobsthals}, these are bounded above by $\log^2(c(2nk)^n) = n^2\log^2(2nk)+\log^2(c)$.
    So, the entries of $MA$ are bounded above by $k\big(n^2\log^2(2nk)+\log^2(c)\big)$, because there is at most one nonzero entry in each row of $A$ by construction.
    Thus, $$\norm{MA\mb{u}}\leq n^2\max_{i,j}|(MA)_{i,j}|\leq n^2k\big(n^2\log^2(2nk)+\log^2(c)\big)$$
    
    Hence we have $$
    \norm{\left(I_n+\frac{MA}{c\cdot\det M}\right)\mb{u}} \leq 1+\frac{kn^4\log^2(2nk)+kn^2\log^2(c)}{|c\cdot\det M|}.
    $$
    Likewise, we see that \begin{equation*}1-\frac{kn^4\log^2(2nk)+kn^2\log^2(c)}{|c\cdot\det M|}\leq \norm{\left(I_n+\frac{MA}{c\cdot\det M}\right)\mb{u}} .\qedhere\end{equation*}
\end{proof}

The following proposition gives a general tool for showing gap-preservation of the reductions in Section \ref{reductions_section}.
Following the methodology laid out by Martin in Theorem 4.7 of \cite{daniel}, we first consider how much multiplication by $M\tilde{M}$ can inflate the gap without invalidating the output and then show that the necessary condition is already met. The argument below can be generalized to hold for any $\ell_p$-norm by changing the squares and roots to be $p^\text{th}$ powers and roots appropriately and choosing $c$ in terms of $p$. We restrict to the more typical $\ell_1,\ell_2,$ and $\ell_\infty$-norms for clarity.

\begin{proposition}\label{gap-preservation}
    Let $c\in\N$, $k$ and $M$ be as in the input of Algorithm \ref{approximate}, $\tilde{M}=c\cdot\adj M +A$ be as in the output of Algorithm \ref{approximate}, and $\gamma=a/b$ be such that $a,b\in N$ with $k\ge a\ge b$. Also, let $\alpha=t/d\in \Q$ be such that $k\ge d$ and $\alpha\leq |\det M|$ and let $\mb{y}\in\mathbb{X}\subset\Z^n$. Then, 
    under the $\ell_1,\ell_2,$ or $\ell_\infty$-norm, if $\norm{\mb{y}}\le \gamma\norm{\mb{x}}$ for all $\mb{x}\in\mathbb{X}$, then $\norm{M\tilde{M}\mb{y}}\le \gamma\norm{M\tilde{M}\mb{x}}$ for all $\mb{x}\in\mathbb{X}.$
\end{proposition}

\begin{proof} 
    Since $\smash{\norm{M\tilde{M}\mb{y}}^2}$ is an integer under the $\ell_1,\ell_2$ or $\ell_\infty$-norm, $\smash{\norm{M\tilde{M}\mb{y}}^2}\leq (\alpha a/b)^2$ is satisfied if $\norm{M\tilde{M}\mb{y}}^2<((ad\alpha)^2+1)/(bd)^2$ is satisfied since there are no integers strictly between $((ad\alpha)^2+1)/(bd)^2$ and $(\alpha a/b)^2$. 
    Then, for some $B\in\Mn(\Z)$, if $$\norm{B\mb{y}}< \gamma\dfrac{\sqrt{a^2d^2\alpha^2 + 1}}{ad\alpha}\norm{B\mb{x}}$$ for all $\mb{x}\in\mathbb{X}$ then $\norm{B\mb{y}}\leq \gamma\norm{B\mb{x}}$ for all $\mb{x}\in\mathbb{X}$.

    So we aim to show that 
    $$\norm{M\tilde{M}\mb{y}}<\gamma\dfrac{\sqrt{a^2d^2\alpha^2 + 1}}{ad\alpha}\norm{M\tilde{M}\mb{x}},$$
    or equivalently,
    $$\frac{\norm{M\tilde{M}\mb{y}}}{\norm{M\tilde{M}\mb{x}}}< \gamma\dfrac{\sqrt{a^2d^2\alpha^2 + 1}}{ad\alpha}.$$

    The multiplier on the right hand side can be bounded below using the hypothesis that $\alpha\le |\det M|$ to get 
    $$\dfrac{\sqrt{a^2d^2\alpha^2 + 1}}{ad\alpha}\ge \dfrac{\sqrt{a^2d^2|\det M|^2 + 1}}{ad|\det M|}.$$

    Since scaling will not affect the ratio of the norms on the left hand side, we consider the operator $I_n+\frac{MA}{c\cdot\det M}$ instead of $M\tilde{M}$.
    
    Applying the bounds from Lemma \ref{opnorm_prop} and the hypothesis that $\norm{\mb{y}}\le \gamma\norm{\mb{x}}$, we get a sufficient condition of 
    $$\frac{1+\frac{kn^4\log^2(2nk)+kn^2\log^2(c)}{|c\cdot\det M|}}{1-\frac{kn^4\log^2(2nk)+kn^2\log^2(c)}{|c\cdot\det M|}}< \dfrac{\sqrt{a^2d^2|\det M|^2 + 1}}{ad|\det M|}.$$

    Rearranging, we get
    
$$\frac{kn^4\log^2(2nk) + kn^2\log^2(c)}{|\det M|}\cdot\frac{\sqrt{a^2d^2|\det M|^2 + 1} + ad|\det M|}{\sqrt{a^2d^2|\det M|^2 + 1} - ad|\det M|} < c.$$

Notice that $$\frac{\sqrt{a^2d^2|\det M|^2 + 1} + ad|\det M|}{\sqrt{a^2d^2|\det M|^2 + 1} - ad|\det M|} < 6a^2d^2|\det M|^2,$$ so it is sufficient for $$6a^2d^2|\det M|kn^4\log^2(2nk) + 6a^2d^2|\det M|kn^2\log^2(c) < c.$$

The left hand side can be bounded above using Hadamard's inequality and the hypotheses that $a,d\le k$ to get
$$6k^{n+5}n^{n+4}\log^2(2nk) + 6k^{n+5}n^{n+2}\log^2(c) < c.$$

We see that $$6k^{n+5}n^{n+4}\log^2(2nk) < c/2 \implies 12k^{n+5}n^{n+4}\log^2(2nk)< c.$$

Since $\log^2(c)\leq c^{2/3}$ for all $c\geq 94$, then $6k^{n+5}n^{n+2}\log^2(c) < c/2$ if $6k^{n+5}n^{n+2}c^{2/3} < c/2$.  Rearranging, we get $1728 k^{3n+15}n^{3n+6} < c.$

So, it is sufficient for $$c > \max\{12k^{n+5}n^{n+4}\log^2(2nk), 1728 k^{3n+15}n^{3n+6} , 94\}.$$ But this is satisfied by the choice of $c$ in line 1 of Algorithm \ref{approximate}.
\end{proof}

\subsection{Time-Complexity}
We now consider the time-complexity of Algorithm \ref{approximate}.

\begin{proposition}\label{approx_runtime}
     Let $k$ and $M$ be as in the input of Algorithm \ref{approximate}, $\tilde{M}$ be as throughout Algorithm \ref{approximate}, and $c$ be the multiplier used in line 1. Algorithm  \ref{approximate} takes $\bigo\big(n^5\log^3(nk)\big)$ integer operations.
\end{proposition}

\begin{proof}
    Computing the adjugate on line 1 takes $\bigo(n^3)$ integer operations.  
     By Lemma \ref{inflation}, the magnitude of entries of $\tilde{M}$ at every point are bounded above by $c(2nk)^n$.
     The Euclidean Algorithm call on line 14 and each call on line 8 takes $$\bigo\big(\log(\min\{\det(B_1),\det(B_2)\})\big)\leq \bigo\big (\log(c^n(2nk)^{n^2+n})\big )=\bigo\big(n^2\log(nk)+n\log(c)\big).$$
    We can see that all the steps on lines 9-13 and 2-6 are constant.
    The matrix inversion and multiplication that produces  the output vector on line 15 will take $\bigo(n^3)$ integer operations.

    Now consider the complexity of the \textbf{for} loop on lines 7-13.
    In each iteration, we compute two determinants that each take $\bigo(i^3)$ integer operations. Notice that we do not recompute these determinants each iteration of the following \textbf{while} loop because of the relationship between the old and new determinants shown in Proposition \ref{approximate_prop}. 

    Theorem \ref{jacobsthals} along with Lemma \ref{solution_lemma} give that the \textbf{while} loop on lines 8-13 will iterate  $\bigo\big(n^2\log^2(nk)+\log^2(c)\big)$ times.
    Multiplying this by the $n-1$ iterations of the \textbf{for} loop and the bound on the Euclidean Algorithm from above and then ignoring smaller terms gives a bound of $\bigo\big(n^5\log^3(nk) + n^2\log^3(c)\big)$ for lines 7-13.

    Evaluating $c$, we see the total number of integer operations is \begin{equation*}\bigo\big(n^5\log^3(nk) + n^2\log^3(1728(nk)^{3n+15})\big)=\bigo\big(n^5\log^3(nk)\big).\qedhere\end{equation*}
\end{proof}

\section{Problem Reductions}\label{reductions_section}

In this section, we give dimension- and gap-preserving reductions from SVP$_\gamma$, SIVP$_\gamma$, and CVP$_\gamma$ to the corresponding problems in SA lattices as well as show that Algorithm \ref{approximate} is optimal in terms of integer inflation.
Each of the following reductions relies heavily on the work done in Algorithm \ref{approximate}. Thus, the three are very similar and Theorem \ref{time_complexity} gives the time-complexity  for all.

\subsection{SVP to SAP}\label{SVP_section}
  First, consider a reduction from SVP$_\gamma$ to SAP$_\gamma$.

\begin{figure}[H]
\begin{algorithm}[H]
\caption{SVP$_\gamma$ to SAP$_\gamma$}\label{SVP_reduction}
    \KwIn{ $k\in \N$, $\gamma=a/b$ where $a,b\in \N$ and $k\geq a\ge b$, and $M\in \Mn(\Z) \,(n\ge 8)$ with $ \det(M)\not = 0$ and $k\geq \max_{i,j} |(M)_{i,j}|$}
    \KwOut{ $\mb{z}_0\in\Z^n$ such that $0<\norm{M\mb{z}_0}\leq \gamma\min_{\mb{z}\in\Z^n}^\times\left\{\norm{M\mb{z}}\right\}$}
     ($\mb{x}, \tilde{M})\gets $ Algorithm \ref{approximate} ($k, M$)\\
     $b_0\gets \SAP(\gamma, \mb{x})$\\
    \Return $M\tilde{M} \{b_0\mb{x}\}$
\end{algorithm}
\end{figure}

\begin{theorem}
    The output of Algorithm \ref{SVP_reduction} solves the original approximate Shortest Vector Problem under the $\ell_1,\ell_2,$ or $\ell_\infty$-norm.
\end{theorem}

\begin{proof}
    Let $\mb{y}= \{b_0\mb{x}\}$ where $b_0$ is as in line 2.
    We see that $M\tilde{M}\mb{y}$ is in $M\Z^n$. 
    Let $\lambda_1$ be the length of the shortest vector in $M\Z^n$.
    By Minkowski's theorem \cite{Minkowski}, we know $\lambda_1\le|\det M|^{1/n}\le |\det M|.$
     Since $\mb{y}$ solves SVP$_\gamma$ in the SA lattice, we can say $\norm{\mb{y}}\le \gamma\norm{\mb{u}}$ for all $\mb{u}$ in the SA lattice.
    Applying Proposition \ref{gap-preservation} with $\alpha =\lambda_1/1$ and $\gamma$ as in the input of Algorithm \ref{SVP_reduction}  gives that $\norm{M\tilde{M}\mb{y}}\le \gamma\norm{M\tilde{M}\mb{u}}$ for all $\mb{u}$ in the SA lattice. Hence $M\tilde{M}\mb{y}$ is a solution to SVP$_\gamma$ in $M\Z^n$.
\end{proof}

In 2002, Dinur showed the following hardness result for the $\ell_\infty$-norm \cite{dinur}:
\begin{theorem}
    SVP$_\gamma$ is NP-hard under the $\ell_\infty$-norm for any $\gamma <n^{c/\log\log n}$ for some constant $c>0.$
\end{theorem}

In 2023, Bennett and Peikert showed the following hardness result for all other $\ell_p$-norms \cite{svp_hard}:
\begin{theorem}
    SVP$_\gamma$ is NP-hard under randomized reductions for the $\ell_p$-norm with any $1\le p < \infty$ and constant $\gamma< 2^{1/p}.$
\end{theorem}

Combining the above theorems gives the following corollary:
\begin{corollary}
    SAP$_\gamma$ is NP-hard under the $\ell_\infty$-norm with $\gamma <n^{c/\log\log n}$ for some constant $c>0$ and is NP-hard under randomized reductions for the $\ell_1$ or $\ell_2$-norm with constant $\gamma< 2$ or $\gamma < 2^{1/2}$, respectively. 
\end{corollary}

\subsection{SIVP to SIAP}\label{SIVP_section}
    Now consider a reduction from SIVP$_\gamma$ to SIAP$_\gamma$.

\begin{figure}[H]
\begin{algorithm}[H]
\caption{SIVP$_\gamma$ to SIAP$_\gamma$}\label{SIVP_reduction}
    \KwIn{ $k\in \N$, $\gamma=a/b$ where $a,b\in \N$ and $k\geq a\ge b$, and $M\in \Mn(\Z) \,(n\ge 8)$ with $ \det(M)\not = 0$ and $k\geq \max_{i,j} |(M)_{i,j}|$}
    \KwOut{ $\mb{v}_1,\mb{v}_2,\dots,\mb{v}_n\in\Z^n$ such that $\mb{v}_i$ are linearly independent and $\max_i\left\{\norm{M\mb{v}_i}\right\}\leq \gamma\lambda_n$}
     $(\mb{x},\tilde{M})\gets $ Algorithm \ref{approximate} ($k, M$)\\
     $b_1,b_2,\dots,b_n\gets \SIAP(\gamma, \mb{x})$\\
    \Return $M\tilde{M}\{b_1\mb{x}\}, M\tilde{M}\{b_2\mb{x}\},\dots, M\tilde{M}\{b_n\mb{x}\}$
\end{algorithm}
\end{figure}

\begin{theorem}\label{SIVP_correct}
    The output of Algorithm \ref{SIVP_reduction} solves the initial approximate Shortest Independent Vector Problem under the $\ell_1,\ell_2,$ or $\ell_\infty$-norm.
\end{theorem}

\begin{proof}
 Let $\mb{y}_i=\{b_i\mb{x}\}$ for each $b_i$ from line 2.
   Then each $M\tilde{M}\mb{y}_i$ will lie in $M\Z^n$ and they will all be linearly independent since $M\tilde{M}$ is full rank. 
   Since $\det (M)I_n$ is a sublattice of $M\Z^n$, we have $\lambda_n \leq |\det M|$. 

    Consider $\mathcal{B}$ the collection of all sets of $n$ linearly independent vectors in the SA lattice. 
    We can say $\max_i\left\{\norm{\mb{y}_i}\right\}\le \gamma \max_i\left\{\norm{\mb{b}_i}\right\}$ for all $\{\mb{b}_1, \mb{b}_2,\dots,\mb{b}_n\}\in\mathcal{B}.$

    Applying Proposition \ref{gap-preservation} with $\alpha = \lambda_n/1$ and $\gamma$ as in the input of Algorithm \ref{SIVP_reduction} gives that $\max_i\left\{\norm{M\tilde{M}\mb{y}_i}\right\}\le \gamma \max_i\left\{\norm{M\tilde{M}\mb{b}_i}\right\}$ for all  $\{\mb{b}_1, \mb{b}_2,\dots,\mb{b}_n\}\in\mathcal{B}.$
    Hence the $M\tilde{M}\mb{y}_i$ are a solution to SIVP$_\gamma$ in $M\Z^n.$
\end{proof}

In 1999, Bl{\"o}mer and Seifert showed the following result about the hardness of SIVP$_\gamma$ \cite{sivp_hard}:
\begin{theorem}
    SIVP$_\gamma$ is NP-hard for any constant $\gamma$.
\end{theorem}

This immediately gives the following corollary:
\begin{corollary}
    SIAP$_\gamma$ is NP-hard for any constant $\gamma$ under the $\ell_1,\ell_2,$ or $\ell_\infty$-norm.
\end{corollary}

\subsection{CVP to CAP}\label{CVP_section}

We now consider a reduction from CVP$_\gamma$ to CAP$_\gamma$ that is almost the same as Algorithm \ref{SVP_reduction}. However, we must transform our target vector before solving CAP$_\gamma$.

\begin{figure}[H]
\begin{algorithm}[H]
\caption{CVP$_\gamma$ to CAP$_\gamma$}\label{CVP_reduction}
    \KwIn{ $k\in \N$, $\gamma=a/b$ where $a,b\in \N$ and $k\geq a\ge b$, $\mb{t}\in\Q^n \,(\mb{t}=(t_1,t_2,\dots,t_n))$ such that $k\ge \textup{lcd}(t_i)$, and $M\in \Mn(\Z) \,(n\ge 8)$ with $ \det(M)\not = 0$ and $k\geq \max_{i,j} |(M)_{i,j}|$}
    \KwOut{ $\mb{z}_0\in\Z^n$ such that $0\leq\norm{M\mb{z}_0-\mb{t}}\leq \gamma\min_{\mb{z}\in\Z^n}\left\{\norm{M\mb{z}-\mb{t}}\right\}$}
     $(\mb{x},\tilde{M})\gets $ Algorithm \ref{approximate} ($k, M$)\\
     $b_0 \gets \CAP(\gamma, \mb{x}, (M\tilde{M})^{-1}\mb{t})$\\
    \Return $M\tilde{M}(\{b_0\mb{x}-\mb{t}\} +\mb{t})$
\end{algorithm}
\end{figure}

\begin{theorem}
    The output of Algorithm \ref{CVP_reduction} solves the initial approximate Closest Vector Problem under the $\ell_1,\ell_2,$ or $\ell_\infty$-norm.
\end{theorem}

\begin{proof}
   
Let $\alpha $ be the minimum distance between the target vector and a lattice vector in $M\Z^n$. $\alpha$ will be rational with denominator equal to $\textup{lcd}(t_i)\leq k$ and since we have $\det(M)I_n$  as a sublattice of $M\Z^n$, then $\alpha\leq |\det M|.$ 

Let $\mb{y}=\{b_0\mb{x}-\mb{t}\} +\mb{t}$ as in line 3. Note that $\{b\mb{x}-\mb{t}\} = b\mb{x}-\mb{t}-\lceil b\mb{x}-\mb{t}\rfloor$, so the solution to CVP$_\gamma$ in the SA lattice is $b_0\mb{x}-\lceil b_0\mb{x}-\mb{t}\rfloor = \{b_0\mb{x}-\mb{t}\} +\mb{t}$, which is $\mb{y}$. We see that $M\tilde{M}\mb{y}$ is in $M\Z^n$ and we can say $\norm{\mb{y} - (M\tilde{M})^{-1}\mb{t}}\le \gamma  \norm{\mb{u}-(M\tilde{M})^{-1}\mb{t}}$ for all $\mb{u}$ in the SA lattice.

Applying Proposition \ref{gap-preservation} with $\alpha = \alpha$ as above and $\gamma$ as in the input of Algorithm \ref{CVP_reduction} gives that
\begin{align*}
    \norm{M\tilde{M}\big(\mb{y}-(M\tilde{M})^{-1}\mb{t}\big)}&\le \gamma\norm{M\tilde{M}\big(\mb{u}-(M\tilde{M})^{-1}\mb{t}\big)}
    \\\norm{M\tilde{M}\mb{y}-\mb{t}}&\le \gamma\norm{M\tilde{M}\mb{u}-\mb{t}}
\end{align*}
for all $\mb{u}$ in the SA lattice.
Hence, $M\tilde{M}\mb{y}$ is a solution to CVP$_\gamma$ in $M\Z^n$.
\end{proof}

In 1997, Arora et al. proved the following about the hardness of CVP$_\gamma$ \cite{cvp_hard}:

\begin{theorem}
    CVP$_\gamma$ is NP-hard for any constant $\gamma$.
\end{theorem}
This immediately gives the following corollary:
\begin{corollary}
    CAP$_\gamma$ is NP-hard for any constant $\gamma$ under the $\ell_1,\ell_2,$ or $\ell_\infty$-norm.
\end{corollary}

As mentioned before, the majority of the computational complexity of Algorithms \ref{SVP_reduction}, \ref{SIVP_reduction} and \ref{CVP_reduction} occurs in the call to Algorithm \ref{approximate}. Thus, the following comes immediately from Proposition \ref{approx_runtime}:

\begin{theorem}\label{time_complexity}
    Let $k$ be as in the input to Algorithms \ref{SVP_reduction}, \ref{SIVP_reduction}, and \ref{CVP_reduction}. Then
Algorithms  \ref{SVP_reduction}, \ref{SIVP_reduction}, and \ref{CVP_reduction} each take $\bigo\big(n^5\log^3(nk)\big)$ integer operations.
\end{theorem}
\begin{proof}
    The only computations in Algorithms \ref{SVP_reduction}, \ref{SIVP_reduction}, and \ref{CVP_reduction} other than calls to Algorithm \ref{approximate} or an oracle  are the multiplications in the output. Each of these takes $\bigo(n^3)$ integer operations. Thus, each algorithm takes $\bigo\big(n^5\log^3(nk)\big)$ integer operations by Proposition \ref{approx_runtime}.
\end{proof}

\subsection{Optimal Integer Inflation}\label{optimal_section}

Consider the following lemma about the covolume of an SA lattice:

\begin{lemma}\label{covol}
    For some $d\in\N$, let $\lat_{\text{SA}}$ be the lattice generated by the columns of $dI_n$ and $\mb{x}\in\Z^n$ where the entries of $\mb{x}$ are collectively coprime. Then the covolume of $\lat_{\text{SA}}$ is $d^{n-1}.$
\end{lemma}
\begin{proof}
    Note that $\lat_{\text{SA}} = d\Z^n + \mb{x}\Z$ and consider the canonical reduction $\phi:\Z^n\rightarrow (\Z/d\Z)^n$.
    Then by the third isomorphism theorem, $$\Z^n/\lat_{\text{SA}}\cong (\Z/d\Z)^n/(\lat_{\text{SA}}/d\Z^n)\cong  (\Z/d\Z)^n/\langle \mb{x}\rangle.$$

    Since the entries of $\mb{x}$ are collectively coprime, they are also collectively coprime  modulo $d$.
    Then $\mb{x}\not= \mb{0}$ over $\Z/d\Z$ and thus $|\langle\mb{x}\rangle| = d.$
    Since $|(\Z/d\Z)^n|=d^n$, 
    \begin{equation*}[\Z^n:\lat_{\text{SA}}]=|\Z^n/\lat_{\text{SA}}| = |(\Z/d\Z)^n/\langle\mb{x}\rangle| = d^{n-1}.\qedhere\end{equation*}
\end{proof}

By approximating a lattice by its SA counterpart, we immediately get simple candidates for approximately short vectors in both lattices.

\begin{lemma}\label{approx_short_lemma}
Suppose we have some polynomial-time algorithm $A$ that takes an instance of SVP$_\gamma$ with input entries bounded by $k$ and outputs an instance of SAP$_\gamma$ with 
entries bounded by $\bigo(k^{n^t})$ along with a matrix $B$ such that if $\mb{y}$ is a solution to the instance of SAP$_\gamma$, then  $MB\mb{y}$ is a solution to the instance of SVP$_\gamma$ with gap-preservation. Then, we can find approximately short vectors in $M\Z^n$ with $\gamma =\bigo( k^{n^{t-1}})$.     
\end{lemma}

\begin{proof}
    Let $\lat_{\text{SA}}$ be the SA lattice output by Algorithm $A$.
    By scaling, we can consider $\lat_{\text{SA}}$ as generated by  the columns of $dI_n$ for some $d\in\N$ and $\mb{x}\in\Z^n$, where the entries of $\mb{x}$ are reduced mod $d$. Then $\mb{x}$ has length less than $d\sqrt{n}$ and each $d\mb{e}_i$ has length $d$.

    By Lemma \ref{covol}, the covolume of $\lat_{\text{SA}}$ is $d^{n-1}$.
    By Minkowski's theorem \cite{Minkowski}, the shortest vector in $\lat_{\text{SA}}$ will have length less than $(d^{n-1})^{1/n} = d/d^{1/n}$.
    So the ratio between the lengths of the vectors in the generating set of $\lat_{\text{SA}}$ and the length of the shortest vector is $\gamma  \le\frac{d}{d/d^{1/n}}=d^{1/n}$.

    Algorithm $A$ ensures that $d=\bigo(k^{n^{t}})$ and thus $\gamma\le \bigo\big((k^{n^{t}})^{1/n}\big) = \bigo(k^{n^{t-1}})$.
    We can transform these generating vectors from $\lat_{\text{SA}}$ to get approximately short vectors in $M\Z^n$ with $\gamma=\bigo(k^{n^{t-1}})$.
    \end{proof}

Were an algorithm to exist that had a sub-quadratic bound on integer inflation, Lemma \ref{approx_short_lemma} says that we would have a method for approximating short vectors that is sub-exponential in $n$. More precisely,

\begin{theorem}
    Suppose there exists some polynomial-time algorithm $A_\varepsilon$ that takes an instance of SVP$_\gamma$ with input entries bounded by $k$ and outputs an instance of SAP$_\gamma$ with 
entries bounded by $\bigo(k^{n^{2-\varepsilon}})$ along with a matrix $B$ such that if $\mb{y}$ is a solution to the instance of SAP$_\gamma$, then  $MB\mb{y}$ is a solution to the instance of SVP$_\gamma$ with gap-preservation. Then, we can find approximately short vectors in $M\Z^n$ with $\gamma = \bigo(k^{n^{1-\varepsilon}})$.     
\end{theorem}

Since Algorithm \ref{approximate} is such an algorithm as in Lemma \ref{approx_short_lemma} and gives an instance of SAP$_\gamma$ with entries bounded by $\bigo\big(c^n(kn)^{n^2}\big)$, it will produce short vectors with exponential approximation quality.
So Algorithm \ref{approximate} is optimal in integer inflation without achieving a sub-exponential in $n$ approximation for shortest vectors.

\bibliographystyle{plain}
\bibliography{bib.bib}
\end{document}